\documentclass[11pt]{amsart}
\usepackage{amsmath,amssymb}
\newtheorem{theorem}{Theorem}[section]

\newtheorem{corollary}{Corollary}[section]
\newtheorem{lemma}{Lemma}[section]
\theoremstyle{definition}
\newtheorem{definition}{Definition}[section]
\newtheorem{remark}{Remark}[section]

\usepackage{latexsym,amssymb}
\begin{document}
\baselineskip=14.5pt
\title[A generalization of  T\'oth identity in the ring of Algebraic Integers]{A generalization of  T\'oth identity in the ring of Algebraic Integers involving a Dirichlet Character} 

\author{Subha Sarkar}
\address{Department of Mathematics, Ramakrishna Mission Vivekananda Educational and Research Institute, Belur Math, Howrah-711 202, West Bengal, India}
\email{subhasarkar2051993@gmail.com}

\begin{abstract}
The $k$-dimensional generalized Euler function $\varphi_k(n)$ is defined to be the number of ordered $k$-tuples $(a_1,a_2,\ldots, a_k) \in \mathbb{N}^k$ with $1\leq a_1,a_2,\ldots, a_k \leq n$ such that both the product $a_1a_2\cdots a_k$ and the sum $a_1+a_2+\cdots+a_k$ are co-prime to $n$. T\'oth proved that the identity

\medskip
\begin{equation*}
\sum_{\substack{a_1,a_2,\ldots, a_k=1 \\ \gcd(a_1a_2\cdots a_k,n)=1\\ \gcd(a_1+a_2+\cdots+a_k,n)=1}}^n \gcd(a_1+a_2+\cdots+a_k-1,n) =\varphi_k(n)\sigma_0(n), \;\; \text{ where } \sigma_s(n) = \sum_{d\mid n}d^s \;\; \text{ holds. }
\end{equation*}

\medskip

This identity can also be viewed as a generalization of Menon's identity. In this article, we generalize this identity to the ring of algebraic integers involving arithmetical functions and Dirichlet characters.
 \end{abstract}

%

%


\maketitle

\section{Introduction}
For a positive integer $n$, the classical Menon's identity \cite{menon} states that
\begin{equation}\label{menon}
\displaystyle\sum_{\substack{a=1 \\ \gcd(a,n)=1}}^n \gcd(a-1,n) =\varphi(n)\sigma_0(n),
\end{equation}
where $\varphi(n)$ is the Euler's totient function and $\sigma_0(n) =\displaystyle \sum_{d \mid n}1$.

\smallskip

Sury \cite{sury} generalized \eqref{menon} and proved that, for every integers $n \geq 1$ and $s\geq 0$,
\begin{equation}\label{sury}
\sum_{\substack{1 \leq a, b_1, b_2, \ldots, b_s \leq n \\ \gcd(a,n)=1}} \gcd(a-1,b_1,\ldots, b_s,n)= \varphi(n)\sigma_s(n), \; \text{ where }\; \sigma_s(n) =\displaystyle \sum_{d \mid n} d^s.
\end{equation}

\smallskip


Li, Hu and Kim \cite{lhk} generalized \eqref{sury} and proved that
\begin{equation}\label{lhk}
\sum_{\substack{1 \leq a, b_1, b_2, \ldots, b_s \leq n \\ \gcd(a,n)=1}} \gcd(a-1,b_1,\ldots, b_s,n)\chi(a)= \varphi(n)\sigma_s\left(\frac{n}{d}\right),
\end{equation}
where $\chi$ is a Dirichlet character modulo $n$ with conductor $d$.

\smallskip

Miguel extended equations \eqref{menon} and \eqref{sury} in \cite{m, m1} from $\mathbb{Z}$ to any residually finite Dedekind domain $\mathcal{D}$. Menon's identity has been generalized in many direction by several authors (cf. \cite{zc,h,lk,toth2,toth3,wj1}).

For every positive integer $k$, T\'{o}th \cite{toth} introduced the $k$-dimensional generalized Euler's totient function $\varphi_{k}$ as,
\begin{equation*}
\varphi_k(n)= \sum_{\substack{a_1, a_2, \ldots, a_k =1 \\ \gcd(a_1\ldots a_k,n)=1 \\ \gcd(a_1+\cdots+a_k,n)=1}}^n 1.
\end{equation*} 

For every positive integer $k$, he proved that the function  $\varphi_k$ is multiplicative and $$\varphi_k(n)= \varphi(n)^k \prod_{p \mid n} \left(1-\frac{1}{p-1}+\frac{1}{(p-1)^2}+\cdots+ \frac{(-1)^{k-1}}{(p-1)^{k-1}}\right).$$
He also proved that the following Menon-type identity
\begin{equation}\label{toth}
\sum_{\substack{a_1, a_2, \ldots, a_k =1 \\ \gcd(a_1\ldots a_k,n)=1 \\ \gcd(a_1+\cdots+a_k,n)=1}}^n \gcd(a_1+\cdots+a_k-1,n) = \varphi_k(n)\sigma_0(n)\;\;\;\;\;\;\;\; \text{ holds. }
\end{equation}

%


Here we note that the function $\phi_2(n)$ was introduced by Arai and Gakuen \cite{ag}, and Carlitz \cite{c} proved the corresponding formula for $\phi_2(n)$. Also for $k=2$, the identity \eqref{toth},
\begin{equation}\label{sr}
\sum_{\substack{a_1, a_2 =1 \\ \gcd(a_1a_2,n)=1 \\ \gcd(a_1+a_2,n)=1}}^n \gcd(a_1+a_2-1,n) = \varphi_2(n)\sigma_0(n)
\end{equation}
was deduced by Sita Ramaiah \cite{sr}.

Recently, Ji and Wang \cite{wj} and Chattopadhyay and Sarkar \cite{js} generalized Sita Ramaiah's identity in the ring of algebraic integers involving a Dirichlet character. Before we state the results of \cite{js} and \cite{wj}, we introduce the fundamental functions defined on the integral ideals of $\mathcal{O}_{K}$. In what follows, for an element $\alpha \in \mathcal{O}_{K}$, the principal ideal $\alpha\mathcal{O}_{K}$ is denoted by $\langle \alpha \rangle$ and $\gcd(\alpha,\mathfrak{a})$ will denote $\gcd(\langle \alpha \rangle,\mathfrak{a})$.

\medskip

\begin{definition}
Let $K$ be an algebraic number field with ring of integers $\mathcal{O}_{K}$ and $I(\mathcal{O}_K)$ the set of all nonzero integral ideals of $\mathcal{O}_K$. Then for any two arithmetical functions $f$ and $g$ on $I(\mathcal{O}_K)$ and for any $\mathfrak{n} \in I(\mathcal{O}_K)$, the Dirichlet convolution ``$\ast$" is defined by $$f\ast g (\mathfrak{n})= \sum_{\mathfrak{d}\mid \mathfrak{n}} f(\mathfrak{d})g(\mathfrak{n}/\mathfrak{d}).$$
\end{definition}

\smallskip

\begin{definition}
Let $K$ be an algebraic number field with ring of integers $\mathcal{O}_{K}$ and $\mathfrak{n} \in I(\mathcal{O}_{K})$. Then we define the following functions on the set $I(\mathcal{O}_{K})$. 
\begin{enumerate}
\item The M\"obius $\mu$ function is defined as 
\begin{equation*}
\mu(\mathfrak{n}) :=
\begin{cases}
    1   & ~ \text{ if } \mathfrak{n} = \mathcal{O}_{K},\\
    (-1)^{t} &  ~ \text{ if } \mathfrak{n} = \displaystyle\prod_{\i = 1}^{t}\mathfrak{p}_{i}; \mbox{ where each } \mathfrak{p}_{i} \subseteq \mathcal{O}_{K} \mbox{ is a prime ideal and } \mathfrak{p}_{i} \neq \mathfrak{p}_{j} \mbox{ for } i \neq j,\\
    0 & ~ \text{ otherwise. }
\end{cases}
\end{equation*}

\medskip

\item The Euler totient function is defined as $$\varphi(\mathfrak{n}) := |(\mathcal{O}_{K}/\mathfrak{n})^{*}| =  N(\mathfrak{n}) \prod_{\mathfrak{p}\mid \mathfrak{n}}\left(1-\frac{1}{N(\mathfrak{p})}\right),$$ where for any ideal $\mathfrak{a} \subseteq \mathcal{O}_K$, $N(\mathfrak{a})$ stands for the absolute norm $|\mathcal{O}_{K}/\mathfrak{a}|$.

\medskip

\item For an integer $s \geq 0$, the function $\sigma_{s}$ is defined by $$\sigma_{s}(\mathfrak{n}) := \displaystyle\sum_{\mathfrak{d} \mid \mathfrak{n}} N(\mathfrak{d})^{s}.$$ 

\item The function $\varphi_{2}(\mathfrak{n})$ is defined as $$\varphi_{2}(\mathfrak{n}) := \sum_{\substack{a,b \in (\mathcal{O}_{K}/\mathfrak{n})^{*} \\ a+b \in (\mathcal{O}_{K}/\mathfrak{n})^{*}}} 1.$$ It is also known that (cf. \cite{wj}), $\phi_{2}(\mathfrak{n}) = \phi(\mathfrak{n})^2 \displaystyle\sum_{\mathfrak{d}\mid \mathfrak{n}}\frac{\mu(\mathfrak{d})}{\phi(\mathfrak{d})}=\phi(\mathfrak{n})^2 \prod_{\mathfrak{p} \mid \mathfrak{n}} \left(1-\frac{1}{N(\mathfrak{p})-1}\right)$.
\end{enumerate}
\end{definition}

\bigskip

\begin{definition}
Let $K$ be an algebraic number field with ring of integers $\mathcal{O}_K$ and $\mathfrak{n} \in I(\mathcal{O}_{K})$. Let $\chi$ be any Dirichlet character modulo $\mathfrak{n}$. 
\begin{enumerate}
\item[(i)] An integral divisor $\mathfrak{d}$ of $\mathfrak{n}$ is said to be an {\it induced modulus} for $\chi$ if $\chi(a)=\chi(b)$ whenever $a,b \in (\mathcal{O}_{K}/\mathfrak{n})^{*}$ and $a \equiv b \pmod{\mathfrak{d}}$. The unique induced modulus $\mathfrak{d}_0$ such that for any induced modulus $\mathfrak{d}$, we have $\mathfrak{d}_0 \mid \mathfrak{d}$, is said to be the {\it conductor} of $\chi$.

\item[(ii)]A character $\chi$ modulo $\mathfrak{n}$ is said to be {\it primitive modulo} $\mathfrak{n}$ if it has no induced modulus $\mathfrak{d}$ with $\mathfrak{d} \neq \mathfrak{n}$. Let $\chi$ be any character modulo $\mathfrak{n}$ with conductor $\mathfrak{d}$. Then there is a unique primitive character $\psi$ modulo $\mathfrak{d}$ such that  $\chi(a)=\psi(a)$ for all $a\in (\mathcal{O}_{K}/\mathfrak{n})^{*}$.
\end{enumerate}
\end{definition}

\medskip

A special case of the result in \cite{js} is the following generalization of Sita Ramaiah's identity:
\begin{equation} \label{js}
\sum_{\substack{ a_1, a_2, a_1+a_2 \in (\mathcal{O}_{K}/\mathfrak{n})^* \\ b_1,b_2,\ldots, b_s \in \mathcal{O}_{K}/\mathfrak{n}}} N(\gcd(a_1+a_2-1,b_1,b_2,\ldots, b_s,\mathfrak{n}))\chi(a_1) =\mu(\mathfrak{d}) \varphi\left(\frac{\mathfrak{n}_{0}^{2}}{\mathfrak{d}}\right) \varphi_2\left(\frac{\mathfrak{n}}{\mathfrak{n}_{0}}\right)\sigma_s\left(\frac{\mathfrak{n}}{\mathfrak{d}}\right),
\end{equation}
where $\chi$ is a Dirichlet character modulo $\mathfrak{n}$ with conductor $\mathfrak{d}$, $\mathfrak{n}_{0} \mid \mathfrak{n}$ is such that $\mathfrak{n}_0$ has the same prime ideal factors as that of $\mathfrak{d}$ and $\gcd\left(\mathfrak{n}_0, \frac{\mathfrak{n}}{\mathfrak{n}_0}\right) = 1$.

In this article, we generalize T\'oth's identity \eqref{toth} to the ring of algebraic integers $\mathcal{O}_K$ concerning arithmetical functions on $I(\mathcal{O}_K)$ and Dirichlet characters. But before that, we need to define the $k$-dimensional generalized Euler function $\varphi_k(\mathfrak{n})$ on $I(\mathcal{O}_K)$.

\medskip

\begin{definition}
Let $K$ be an algebraic number field with ring of integers $\mathcal{O}_{K}$ and $\mathfrak{n} \in I(\mathcal{O}_{K})$. For any positive integer $k$, we define the $k$-dimensional generalized Euler function $\varphi_k$ as, $$\varphi_{k}(\mathfrak{n}) := \sum_{\substack{a_1,a_2,\ldots, a_k \in (\mathcal{O}_{K}/\mathfrak{n})^{*} \\ a_1+a_2+\cdots+a_k \in (\mathcal{O}_{K}/\mathfrak{n})^{*}}} 1.$$ Note that $\varphi_1(\mathfrak{n})=\varphi(\mathfrak{n})$. 
\end{definition}

The main results of this article are as follows.

\begin{theorem}\label{MAIN_TH1}
Let $K$ be an algebraic number field with ring of integers $\mathcal{O}_{K}$ and let $\mathfrak{n}$ be a non-zero ideal in $\mathcal{O}_{K}$. Then for every positive integer $k$, we have $$\varphi_{k}(\mathfrak{n}) = \varphi(\mathfrak{n})^k \prod_{\mathfrak{p} \mid \mathfrak{n}} \left(1-\frac{1}{N(\mathfrak{p})-1}+\frac{1}{(N(\mathfrak{p})-1)^2}+\cdots+ \frac{(-1)^{k-1}}{(N(\mathfrak{p})-1)^{k-1}}\right).$$
\end{theorem}

\medskip

\begin{theorem}\label{MAIN-TH}
Let $K$ be an algebraic number field with ring of integers $\mathcal{O}_K$ and let $\mathfrak{n}$ be a non-zero ideal in $\mathcal{O}_{K}$. Let $\chi$ be a Dirichlet character modulo $\mathfrak{n}$ with conductor $\mathfrak{d}$. Then for a fixed element $r \in \mathcal{O}_{K}$ with $(r, \mathfrak{n})=1$ and for any arithmetical function $f$ on $I(\mathcal{O}_K)$, we have
\begin{multline}\label{meq}
 \sum_{\substack{ a_1, a_2,\ldots, a_k \in (\mathcal{O}_{K}/\mathfrak{n})^*\\a_1+a_2+\cdots+a_k \in (\mathcal{O}_{K}/\mathfrak{n})^* \\ b_1,b_2,\ldots, b_s \in \mathcal{O}_{K}/\mathfrak{n}}} f(\gcd(a_1+a_2+\cdots+a_k-r,b_1,b_2,\ldots, b_s,\mathfrak{n}))\chi(a_1)\\ = \mu(\mathfrak{d})^{k-1}\psi(r)\varphi\left(\frac{\mathfrak{n}_{0}^{k}}{\mathfrak{d}^{k-1}}\right)\varphi_k\left(\frac{\mathfrak{n}}{\mathfrak{n}_{0}}\right) \sum_{\substack{\mathfrak{d}\mid \mathfrak{e}\mid \mathfrak{n}\\ \mathfrak{e}\mid b_1,\ldots,\mathfrak{e}\mid b_s}}\frac{(\mu \ast f)(\mathfrak{e})}{\varphi(\mathfrak{e})},
\end{multline}

where $\psi$ is the primitive character modulo $\mathfrak{d}$ that induces $\chi$, $\mathfrak{n}_{0} \mid \mathfrak{n}$ is such that $\mathfrak{n}_0$ has the same prime ideal factors as that of $\mathfrak{d}$ and $\gcd\left(\mathfrak{n}_0, \frac{\mathfrak{n}}{\mathfrak{n}_0}\right) = 1$.
\end{theorem}

Let $f$ be the norm function $f(\mathfrak{n})=N(\mathfrak{n})$ in Theorem \ref{MAIN-TH}. Then we get the following

\medskip

\begin{corollary}\label{f=N}
Let $K$ be an algebraic number field with ring of integers $\mathcal{O}_K$ and let $\mathfrak{n}$ be a non-zero ideal in $\mathcal{O}_{K}$. Let $\chi$ be a Dirichlet character modulo $\mathfrak{n}$ with conductor $\mathfrak{d}$. Then for a fixed element $r \in \mathcal{O}_{K}$ with $(r, \mathfrak{n})=1$, we have
\begin{multline*}
 \sum_{\substack{ a_1, a_2,\ldots, a_k \in (\mathcal{O}_{K}/\mathfrak{n})^*\\a_1+a_2+\cdots+a_k \in (\mathcal{O}_{K}/\mathfrak{n})^* \\ b_1,b_2,\ldots, b_s \in \mathcal{O}_{K}/\mathfrak{n}}} N(\gcd(a_1+a_2+\cdots+a_k-r,b_1,b_2,\ldots, b_s,\mathfrak{n}))\chi(a_1)\\ = \mu(\mathfrak{d})^{k-1}\psi(r)\varphi\left(\frac{\mathfrak{n}_{0}^{k}}{\mathfrak{d}^{k-1}}\right)\varphi_k\left(\frac{\mathfrak{n}}{\mathfrak{n}_{0}}\right)\sigma_s\left(\frac{\mathfrak{n}}{\mathfrak{d}}\right),
\end{multline*}

where $\psi$ is the primitive character modulo $\mathfrak{d}$ that induces $\chi$, $\mathfrak{n}_{0} \mid \mathfrak{n}$ is such that $\mathfrak{n}_0$ has the same prime ideal factors as that of $\mathfrak{d}$ and $\gcd\left(\mathfrak{n}_0, \frac{\mathfrak{n}}{\mathfrak{n}_0}\right) = 1$.
\end{corollary}

In the special case of $r=1$, Corollary \ref{f=N} reduces to the following generalization of the equation \eqref{js}.

\begin{corollary}\label{r=1}
Let $K$ be an algebraic number field with ring of integers $\mathcal{O}_K$ and let $\mathfrak{n}$ be a non-zero ideal in $\mathcal{O}_{K}$. Let $\chi$ be a Dirichlet character modulo $\mathfrak{n}$ with conductor $\mathfrak{d}$. Then
\begin{multline*}
 \sum_{\substack{ a_1, a_2,\ldots, a_k \in (\mathcal{O}_{K}/\mathfrak{n})^*\\a_1+a_2+\cdots+a_k \in (\mathcal{O}_{K}/\mathfrak{n})^* \\ b_1,b_2,\ldots, b_s \in \mathcal{O}_{K}/\mathfrak{n}}} N(\gcd(a_1+a_2+\cdots+a_k-1,b_1,b_2,\ldots, b_s,\mathfrak{n}))\chi(a_1)\\ = \mu(\mathfrak{d})^{k-1}\varphi\left(\frac{\mathfrak{n}_{0}^{k}}{\mathfrak{d}^{k-1}}\right)\varphi_k\left(\frac{\mathfrak{n}}{\mathfrak{n}_{0}}\right) \sigma_s\left(\frac{\mathfrak{n}}{\mathfrak{d}}\right),
\end{multline*}

where $\mathfrak{n}_{0} \mid \mathfrak{n}$ is such that $\mathfrak{n}_0$ has the same prime ideal factors as that of $\mathfrak{d}$ and $\gcd\left(\mathfrak{n}_0, \frac{\mathfrak{n}}{\mathfrak{n}_0}\right) = 1$.
\end{corollary}

\begin{remark}
Note that, for $K = \mathbb{Q}$ and $\chi={\bf 1}$, the trivial character, then Corollary \ref{r=1} reduces to
\begin{equation*}
\sum_{\substack{a_1, a_2, \ldots, a_k,b_1,\ldots,b_s =1 \\ \gcd(a_1\ldots a_k,n)=1 \\ \gcd(a_1+\cdots+a_k,n)=1}}^n \gcd(a_1+\cdots+a_k-1,b_1,\ldots,b_s,n) = \varphi_k(n)\sigma_s(n),
\end{equation*} 
which is a generalization of T\'oth's identity \eqref{toth}.
\end{remark}




\section{Preliminaries}
In this section, we fix an algebraic number field $K$ with ring of integers $\mathcal{O}_{K}$ and a non-zero ideal $\mathfrak{n}$ in $\mathcal{O}_{K}$. We need the following lemmas that have been proved in \cite{wj1}, and we record it here.

\begin{lemma}\cite{wj1}\label{1}
Let $\mathfrak{a}$ be a nonzero ideal in $\mathcal{O}_K$ such that $\mathfrak{a}\mid \mathfrak{n}$. Then for any $u\in \mathcal{O}_{K}$, we have
$$\sum_{\substack{a \in (\mathcal{O}_{K}/\mathfrak{n})^*\\ a \equiv u \pmod{\mathfrak{a}}}}1=
\begin{cases}
\frac{\varphi(\mathfrak{n})}{\varphi(\mathfrak{a})} & \text{ if } (u, \mathfrak{a})=1 \\
0 & \text{ otherwise. }
\end{cases}
$$
\end{lemma}

\medskip

\begin{lemma}\cite{wj1}\label{2}
Let $\mathfrak{a}$ and $\mathfrak{b}$ be two nonzero ideals in $\mathcal{O}_K$ such that $\mathfrak{a}\mid \mathfrak{n}$ and $\mathfrak{b}\mid \mathfrak{n}$. Then for any $u,v\in \mathcal{O}_{K}$, we have
$$\sum_{\substack{a \in (\mathcal{O}_{K}/\mathfrak{n})^*\\ a \equiv u \pmod{\mathfrak{a}}\\a \equiv v \pmod{\mathfrak{b}}}}1=
\begin{cases}
\frac{\varphi(\mathfrak{n})}{\varphi([\mathfrak{a},\mathfrak{b}])} & \text{ if } (u, \mathfrak{a})=(v, \mathfrak{b})=1 \text{ and } u \equiv v \pmod{(\mathfrak{a}, \mathfrak{b})} \\
0 & \text{ otherwise. }
\end{cases}
$$
\end{lemma}

\medskip

\begin{lemma}\cite{wj1}\label{chr}
Let $\psi$ a primitive character modulo $\mathfrak{n}$ and $r \in \mathcal{O}_K$ with $(r, \mathfrak{n})=1$. Then for any $\mathfrak{d}\mid \mathfrak{n}$, $$\sum_{\substack{a \in (\mathcal{O}_{K}/\mathfrak{n})^*\\ a\equiv r \pmod{\mathfrak{d}}}} \psi(a)\neq 0  \;\;\;\;\text{ if and only if }\;\;\;\; \mathfrak{d}=\mathfrak{n}.$$ In particular, if $\mathfrak{d}= \mathfrak{n}$, then $$\sum_{\substack{a \in (\mathcal{O}_{K}/\mathfrak{n})^*\\ a\equiv r \pmod{\mathfrak{d}}}} \psi(a) = \psi(r).$$
\end{lemma}

\section{Proof of Theorem \ref{MAIN_TH1}}
Let $K$ be an algebraic number field with ring of integers $\mathcal{O}_K$ and $\mathfrak{m}, \mathfrak{n}$  two nonzero ideals in $\mathcal{O}_K$. For every integer $k\geq 1$, we define the following more general function than $\varphi_k(\mathfrak{n})$: $$\varphi_{k}(\mathfrak{n}, \mathfrak{m}) := \sum_{\substack{a_1,a_2,\ldots, a_k \in (\mathcal{O}_{K}/\mathfrak{n})^{*} \\ a_1+a_2+\cdots+a_k \in (\mathcal{O}_{K}/\mathfrak{m})^{*}}} 1.$$ If $\mathfrak{m}=\mathfrak{n}$, then clearly $\varphi_{k}(\mathfrak{n}, \mathfrak{n})=\varphi_k(\mathfrak{n})$. To prove Theorem \ref{MAIN_TH1}, we need the following recursion lemma.

\medskip

\begin{lemma}\label{prec}
Let $k\geq 2$ and $\mathfrak{m} \mid \mathfrak{n}$. Then $$\varphi_{k}(\mathfrak{n}, \mathfrak{m}) = \varphi(\mathfrak{n}) \sum_{\mathfrak{d}\mid \mathfrak{m}}\frac{\mu(\mathfrak{d})}{\varphi(\mathfrak{d})}\varphi_{k-1}(\mathfrak{n}, \mathfrak{d}).$$ 
\end{lemma}

\begin{proof}
We have
\begin{eqnarray*}
\varphi_{k}(\mathfrak{n}, \mathfrak{m}) = \sum_{\substack{a_1,a_2,\ldots, a_k \in (\mathcal{O}_{K}/\mathfrak{n})^{*} \\ a_1+a_2+\cdots+a_k \in (\mathcal{O}_{K}/\mathfrak{m})^{*}}} 1 &=& \sum_{\substack{a_1,a_2,\ldots, a_k \in (\mathcal{O}_{K}/\mathfrak{n})^{*}}} \sum_{\mathfrak{d}\mid (a_1+a_2+\cdots+a_k, \mathfrak{m})} \mu(\mathfrak{d})\\
&=& \sum_{\mathfrak{d}\mid \mathfrak{m}}  \mu(\mathfrak{d}) \sum_{\substack{a_1,a_2,\ldots, a_k \in (\mathcal{O}_{K}/\mathfrak{n})^{*}\\a_1+a_2+\cdots+ a_k \equiv 0 \pmod{\mathfrak{d}}}}1\\
&=& \sum_{\mathfrak{d}\mid \mathfrak{m}}  \mu(\mathfrak{d})\sum_{\substack{a_1,a_2,\ldots, a_{k-1} \in (\mathcal{O}_{K}/\mathfrak{n})^{*}}} \sum_{\substack{a_k \in (\mathcal{O}_{K}/\mathfrak{n})^{*}\\a_k \equiv -a_1-a_2-\cdots- a_{k-1} \pmod{\mathfrak{d}}}}1
\end{eqnarray*}
By using Lemma \ref{1} we get,
\begin{eqnarray*}
\varphi_{k}(\mathfrak{n}, \mathfrak{m}) &=& \sum_{\mathfrak{d}\mid \mathfrak{m}}  \mu(\mathfrak{d})\sum_{\substack{a_1,a_2,\ldots, a_{k-1} \in (\mathcal{O}_{K}/\mathfrak{n})^{*}\\a_1+a_2+\cdots+ a_{k-1}\in (\mathcal{O}_{K}/\mathfrak{d})^{*}}}\frac{\varphi(\mathfrak{n})}{\varphi(\mathfrak{d})}\\
&=&\varphi(\mathfrak{n})\sum_{\mathfrak{d}\mid \mathfrak{m}}\frac{\mu(\mathfrak{d})}{\varphi(\mathfrak{d})} \sum_{\substack{a_1,a_2,\ldots, a_{k-1} \in (\mathcal{O}_{K}/\mathfrak{n})^{*}\\a_1+a_2+\cdots+ a_{k-1}\in (\mathcal{O}_{K}/\mathfrak{d})^{*}}}1\\
&=& \varphi(\mathfrak{n})\sum_{\mathfrak{d}\mid \mathfrak{m}}\frac{\mu(\mathfrak{d})}{\varphi(\mathfrak{d})}\varphi_{k-1}(\mathfrak{n}, \mathfrak{d}).
\end{eqnarray*}
\end{proof}
{\noindent \bf Proof of Theorem \ref{MAIN_TH1}.}
Let $k\geq 1$ and $\mathfrak{m}$ and $\mathfrak{n}$ be two nonzero ideals of $\mathcal{O}_K$ such that $\mathfrak{m} \mid \mathfrak{n}$. Then we show the following more general result $$\varphi_{k}(\mathfrak{n},\mathfrak{m}) = \varphi(\mathfrak{n})^k \prod_{\mathfrak{p} \mid \mathfrak{m}} \left(1-\frac{1}{N(\mathfrak{p})-1}+\frac{1}{(N(\mathfrak{p})-1)^2}+\cdots+ \frac{(-1)^{k-1}}{(N(\mathfrak{p})-1)^{k-1}}\right),$$
by induction on $k$. If $k=1$, then $\varphi_{1}(\mathfrak{n},\mathfrak{m})=\varphi(\mathfrak{n})$, by the definition. Let $k \geq 2$, and assume that the result is true for $k-1$. Using Lemma \ref{prec}, we have
\begin{eqnarray*}
\varphi_{k}(\mathfrak{n}, \mathfrak{m}) &=& \varphi(\mathfrak{n}) \sum_{\mathfrak{d}\mid \mathfrak{m}}\frac{\mu(\mathfrak{d})}{\varphi(\mathfrak{d})}\varphi_{k-1}(\mathfrak{n}, \mathfrak{d})\\
&=& \varphi(\mathfrak{n}) \sum_{\mathfrak{d}\mid \mathfrak{m}}\frac{\mu(\mathfrak{d})}{\varphi(\mathfrak{d})} \varphi(\mathfrak{n})^{k-1}\prod_{\mathfrak{p} \mid \mathfrak{d}} \left(1-\frac{1}{N(\mathfrak{p})-1}+\frac{1}{(N(\mathfrak{p})-1)^2}+\cdots+ \frac{(-1)^{k-2}}{(N(\mathfrak{p})-1)^{k-2}}\right)\\
&=& \varphi(\mathfrak{n})^k \prod_{\mathfrak{p} \mid \mathfrak{m}}1+\frac{\mu(\mathfrak{p})}{\varphi(\mathfrak{p})} \left(1-\frac{1}{N(\mathfrak{p})-1}+\frac{1}{(N(\mathfrak{p})-1)^2}+\cdots+ \frac{(-1)^{k-2}}{(N(\mathfrak{p})-1)^{k-2}}\right)\\
&=& \varphi(\mathfrak{n})^k \prod_{\mathfrak{p} \mid \mathfrak{m}}\left(1-\frac{1}{N(\mathfrak{p})-1}+\frac{1}{(N(\mathfrak{p})-1)^2}+\cdots+ \frac{(-1)^{k-1}}{(N(\mathfrak{p})-1)^{k-1}}\right),
\end{eqnarray*} 
which proves the result.

Now if we take $\mathfrak{m}=\mathfrak{n}$, then we have the proof of Theorem \ref{MAIN_TH1}. $\hfill\Box$

\section{Proof of Theorem \ref{MAIN-TH}}
Let $K$ be an algebraic number field with ring of integers $\mathcal{O}_K$ and $\mathfrak{n}$ a nonzero ideal in $\mathcal{O}_K$. We are given that $\chi$ is a Dirichlet character modulo $\mathfrak{n}$ with conductor $\mathfrak{d}$ and $\psi$ is the primitive character modulo $\mathfrak{d}$ that induces $\chi$. Let $a_1 \in (\mathcal{O}_{K}/\mathfrak{n})^{*}$ and $u \in (\mathcal{O}_{K}/\mathfrak{d})^{*}$ be such that $a_1\equiv u \pmod{\mathfrak{d}}$. Since $\psi$ is the primitive character modulo $\mathfrak{d}$ that induces $\chi$, we have $\chi(a_1)=\psi(a_1)=\psi(u)$. Let $M_{\chi,k}(\mathfrak{n})$ denote the sum in left hand side of the equation \eqref{meq}. Then
\begin{eqnarray*}
M_{\chi,k}(\mathfrak{n})&=&\sum_{\substack{ a_1, a_2,\ldots, a_k \in (\mathcal{O}_{K}/\mathfrak{n})^* \\ a_1+a_2+\cdots+a_k \in (\mathcal{O}_{K}/\mathfrak{n})^*\\ b_1,b_2,\ldots, b_s \in \mathcal{O}_{K}/\mathfrak{n}}} f(\gcd(a_1+a_2+\cdots+a_k-r,b_1,b_2,\ldots, b_s,\mathfrak{n}))\chi(a_1)\nonumber\\
&=& \sum_{u \in (\mathcal{O}_{K}/\mathfrak{d})^{*}}\psi(u)\sum_{\substack{ a_1, a_2,\ldots, a_k  \in (\mathcal{O}_{K}/\mathfrak{n})^*\\a_1+a_2+\cdots+a_k \in (\mathcal{O}_{K}/\mathfrak{n})^*\\a_1\equiv u \pmod{\mathfrak{d}}\\ b_1,b_2,\ldots, b_s \in \mathcal{O}_{K}/\mathfrak{n}}} f(\gcd(a_1+a_2+\cdots+a_k-r,b_1,b_2,\ldots, b_s,\mathfrak{n}))\nonumber
\end{eqnarray*}
By using the convolution identity $f=(\mu \ast f)\ast {\bf 1}$, we get
\begin{eqnarray}\label{sum}
M_{\chi,k}(\mathfrak{n})&=&\sum_{u \in (\mathcal{O}_{K}/\mathfrak{d})^{*}}\psi(u)\sum_{\substack{ a_1, a_2,\ldots, a_k  \in (\mathcal{O}_{K}/\mathfrak{n})^*\\a_1+a_2+\cdots+a_k \in (\mathcal{O}_{K}/\mathfrak{n})^*\\a_1\equiv u \pmod{\mathfrak{d}}\\ b_1,b_2,\ldots, b_s \in \mathcal{O}_{K}/\mathfrak{n}}}\sum_{\mathfrak{e}\mid (a_1+a_2+\cdots+a_k-r,b_1,b_2,\ldots, b_s,\mathfrak{n})} (\mu \ast f)(\mathfrak{e})\nonumber\\
&=& \sum_{\substack{\mathfrak{e}\mid \mathfrak{n}}}(\mu \ast f)(\mathfrak{e})\sum_{\substack{b_1,b_2,\ldots, b_s \in \mathcal{O}_{K}/\mathfrak{n}\\\mathfrak{e} \mid b_1,\ldots, \mathfrak{e} \mid b_s}}\;\;\sum_{u \in (\mathcal{O}_{K}/\mathfrak{d})^{*}}\psi(u)\sum_{\substack{ a_1, a_2,\ldots, a_k  \in (\mathcal{O}_{K}/\mathfrak{n})^*\\a_1+a_2+\cdots+a_k \in (\mathcal{O}_{K}/\mathfrak{n})^*\\a_1\equiv u \pmod{\mathfrak{d}}\\a_1+a_2+\cdots+a_k \equiv r \pmod{\mathfrak{e}}}}1\nonumber\\
&=&\sum_{\substack{\mathfrak{e}\mid \mathfrak{n}}}(\mu \ast f)(\mathfrak{e})\sum_{\substack{b_1,b_2,\ldots, b_s \in \mathcal{O}_{K}/\mathfrak{n}\\\mathfrak{e} \mid b_1,\ldots, \mathfrak{e} \mid b_s}}\;\; \sum_{u \in (\mathcal{O}_{K}/\mathfrak{d})^{*}}\psi(u)\sum_{\substack{ a_1, a_2,\ldots, a_k  \in (\mathcal{O}_{K}/\mathfrak{n})^*\\a_1\equiv u \pmod{\mathfrak{d}}\\a_1+a_2+\cdots+a_k \equiv r \pmod{\mathfrak{e}}}}\sum_{\mathfrak{g}\mid(a_1+a_2+\cdots+a_k, \mathfrak{n})}\mu(\mathfrak{g})\nonumber\\
&=&\sum_{\substack{\mathfrak{e}\mid \mathfrak{n}}}(\mu \ast f)(\mathfrak{e})\sum_{\substack{b_1,b_2,\ldots, b_s \in \mathcal{O}_{K}/\mathfrak{n}\\\mathfrak{e} \mid b_1,\ldots, \mathfrak{e} \mid b_s}}\;\;\sum_{u \in (\mathcal{O}_{K}/\mathfrak{d})^{*}}\psi(u)\;\;\sum_{\mathfrak{g}\mid \mathfrak{n}} \mu(\mathfrak{g})\sum_{\substack{ a_1, a_2,\ldots, a_k  \in (\mathcal{O}_{K}/\mathfrak{n})^*\\a_1\equiv u \pmod{\mathfrak{d}}\\a_1+a_2+\cdots+a_k \equiv r \pmod{\mathfrak{e}}\\a_1+a_2+\cdots+a_k \equiv 0 \pmod{\mathfrak{g}}}}1 \nonumber\\
&=&\sum_{\substack{\mathfrak{e}\mid \mathfrak{n}}}(\mu \ast f)(\mathfrak{e})\sum_{\substack{b_1,b_2,\ldots, b_s \in \mathcal{O}_{K}/\mathfrak{n}\\\mathfrak{e} \mid b_1,\ldots, \mathfrak{e} \mid b_s}}\;\;\sum_{\mathfrak{g}\mid \mathfrak{n}} \mu(\mathfrak{g})\sum_{u \in (\mathcal{O}_{K}/\mathfrak{d})^{*}}\psi(u)N_k(\mathfrak{n}, \mathfrak{e}, \mathfrak{g}, \mathfrak{d},u)
\end{eqnarray}

where $N_k(\mathfrak{n}, \mathfrak{e}, \mathfrak{g}, \mathfrak{d},u)=\displaystyle \sum_{\substack{ a_1, a_2,\ldots, a_k  \in (\mathcal{O}_{K}/\mathfrak{n})^*\\a_1\equiv u \pmod{\mathfrak{d}}\\a_1+a_2+\cdots+a_k \equiv r \pmod{\mathfrak{e}}\\a_1+a_2+\cdots+a_k \equiv 0 \pmod{\mathfrak{g}}}}1$.

\medskip

We now evaluate the sum $N_k(\mathfrak{n}, \mathfrak{e}, \mathfrak{g}, \mathfrak{d},u)$, where $\mathfrak{e}\mid \mathfrak{n}$, $\mathfrak{g}\mid \mathfrak{n}$ and $u \in (\mathcal{O}_{K}/\mathfrak{d})^{*}$ be fixed. Since $(r, \mathfrak{n})=1$, if $(\mathfrak{e}, \mathfrak{g}) >1$, then $N_k(\mathfrak{n}, \mathfrak{e}, \mathfrak{g}, \mathfrak{d},u)=0$, the empty sum. Hence we assume that $(\mathfrak{e}, \mathfrak{g}) =1$.

If $k=1$, using Lemma \ref{2}, we get
\begin{equation}\label{k=1}
N_1(\mathfrak{n}, \mathfrak{e}, \mathfrak{g}, \mathfrak{d},u)=\sum_{\substack{ a_1 \in (\mathcal{O}_{K}/\mathfrak{n})^*\\a_1\equiv u \pmod{\mathfrak{d}}\\a_1 \equiv r \pmod{\mathfrak{e}}\\a_1 \equiv 0 \pmod{\mathfrak{g}}}}1=
\begin{cases}
\frac{\varphi(\mathfrak{n})}{\varphi([\mathfrak{d},\mathfrak{e}])} & \text{ if } \mathfrak{g}=1, (u, \mathfrak{d})=(r, \mathfrak{e})=1 \text{ and } u \equiv r \pmod{(\mathfrak{d}, \mathfrak{e})} \\
0 & \text{ otherwise. }
\end{cases}
\end{equation}


\smallskip

\begin{lemma}(Recursion formula for $N_k(\mathfrak{n}, \mathfrak{e}, \mathfrak{g}, \mathfrak{d},u)$)\label{rec}
Let $\mathfrak{e}$ and $\mathfrak{g}$ be two ideals in $\mathcal{O}_K$ such that $\mathfrak{e}\mid \mathfrak{n}$, $\mathfrak{g}\mid \mathfrak{n}$ with $(\mathfrak{e}, \mathfrak{g}) =1$. Then for every integer $k\geq 2$ and an element $u \in (\mathcal{O}_{K}/\mathfrak{d})^{*}$, we have $$N_k(\mathfrak{n}, \mathfrak{e}, \mathfrak{g}, \mathfrak{d},u)=\frac{\varphi(\mathfrak{n})}{\varphi(\mathfrak{e})\varphi(\mathfrak{g})}\sum_{\mathfrak{j}\mid \mathfrak{e}}\mu(\mathfrak{j})\sum_{\mathfrak{t}\mid \mathfrak{g}}\mu(\mathfrak{t})N_{k-1}(\mathfrak{n}, \mathfrak{j}, \mathfrak{t}, \mathfrak{d},u).$$
\end{lemma}
\begin{proof}
We have
\begin{eqnarray*}
N_k(\mathfrak{n}, \mathfrak{e}, \mathfrak{g}, \mathfrak{d},u)= \sum_{\substack{ a_1, a_2,\ldots, a_k  \in (\mathcal{O}_{K}/\mathfrak{n})^*\\a_1\equiv u \pmod{\mathfrak{d}}\\a_1+a_2+\cdots+a_k \equiv r \pmod{\mathfrak{e}}\\a_1+a_2+\cdots+a_k \equiv 0 \pmod{\mathfrak{g}}}}1 &=& \sum_{\substack{ a_1, a_2,\ldots, a_{k-1}  \in (\mathcal{O}_{K}/\mathfrak{n})^*\\a_1\equiv u \pmod{\mathfrak{d}}}}\sum_{\substack{a_k  \in (\mathcal{O}_{K}/\mathfrak{n})^*\\a_k \equiv r-a_1-a_2-\cdots -a_{k-1} \pmod{\mathfrak{e}}\\a_k \equiv -a_1-a_2-\cdots -a_{k-1} \pmod{\mathfrak{g}}}}1.
\end{eqnarray*}

\smallskip

Since $(\mathfrak{e}, \mathfrak{g}) =1$, using Lemma \ref{2}, we get
\begin{eqnarray*}
N_k(\mathfrak{n}, \mathfrak{e}, \mathfrak{g}, \mathfrak{d},u)&=& \sum_{\substack{ a_1, a_2,\ldots, a_{k-1}  \in (\mathcal{O}_{K}/\mathfrak{n})^*\\a_1\equiv u \pmod{\mathfrak{d}}\\(a_1+a_2+\cdots +a_{k-1}-r, \mathfrak{e})=1\\(a_1+a_2+\cdots +a_{k-1}, \mathfrak{g})=1}}\frac{\varphi(\mathfrak{n})}{\varphi([\mathfrak{e},\mathfrak{g}])}\\
&=& \frac{\varphi(\mathfrak{n})}{\varphi(\mathfrak{e}) \varphi(\mathfrak{g})}\sum_{\substack{ a_1, a_2,\ldots, a_{k-1}  \in (\mathcal{O}_{K}/\mathfrak{n})^*\\a_1\equiv u \pmod{\mathfrak{d}}}}\;\;\sum_{\mathfrak{j}\mid (a_1+a_2+\cdots +a_{k-1}-r, \mathfrak{e})}\mu(\mathfrak{j}) \sum_{\mathfrak{t}\mid (a_1+a_2+\cdots +a_{k-1}, \mathfrak{g})} \mu(\mathfrak{t})\\
&=& \frac{\varphi(\mathfrak{n})}{\varphi(\mathfrak{e}) \varphi(\mathfrak{g})} \sum_{\mathfrak{j}\mid \mathfrak{e}} \mu(\mathfrak{j})\;\;\sum_{\mathfrak{t}\mid \mathfrak{g}} \mu(\mathfrak{t}) \sum_{\substack{ a_1, a_2,\ldots, a_{k-1}  \in (\mathcal{O}_{K}/\mathfrak{n})^*\\a_1\equiv u \pmod{\mathfrak{d}}\\a_1+a_2+\cdots+a_{k-1} \equiv r \pmod{\mathfrak{j}}\\a_1+a_2+\cdots+a_{k-1} \equiv 0 \pmod{\mathfrak{t}}}}1\\
&=& \frac{\varphi(\mathfrak{n})}{\varphi(\mathfrak{e}) \varphi(\mathfrak{g})} \sum_{\mathfrak{j}\mid \mathfrak{e}} \mu(\mathfrak{j})\;\;\sum_{\mathfrak{t}\mid \mathfrak{g}} \mu(\mathfrak{t})N_{k-1}(\mathfrak{n}, \mathfrak{j}, \mathfrak{t}, \mathfrak{d},u).
\end{eqnarray*}
\end{proof}

\begin{lemma}\label{prodprime}
Let $\mathfrak{e}$ and $\mathfrak{g}$ be two nonzero ideals in $\mathcal{O}_K$ such that $\mathfrak{e}\mid \mathfrak{n}$, $\mathfrak{g}\mid \mathfrak{n}$ with $(\mathfrak{e}, \mathfrak{g}) =1$. Then for every integer $k \geq 2$, 
\begin{multline*}
\sum_{u \in (\mathcal{O}_{K}/\mathfrak{d})^{*}}\psi(u)N_k(\mathfrak{n}, \mathfrak{e}, \mathfrak{g}, \mathfrak{d},u)=
\frac{\varphi(\mathfrak{n})^k \psi(r)\mu(\mathfrak{d})^{k-1}}{\varphi(\mathfrak{e})\varphi(\mathfrak{g})}\prod_{\mathfrak{p}\mid \mathfrak{d}}\frac{1}{(N(\mathfrak{p})-1)^{k-1}}\\ \prod_{\substack{\mathfrak{p}\mid \mathfrak{e}\\\mathfrak{p}\nmid \mathfrak{d}}}\left(1- \frac{1}{N(\mathfrak{p})-1}+\cdots+\frac{(-1)^{k-1}}{(N(\mathfrak{p})-1)^{k-1}} \right) \prod_{\substack{\mathfrak{p}\mid \mathfrak{g}}}\left(1- \frac{1}{N(\mathfrak{p})-1}+\cdots+\frac{(-1)^{k-2}}{(N(\mathfrak{p})-1)^{k-2}} \right),
\end{multline*}
if $\mathfrak{d}\mid \mathfrak{e}$. Otherwise, the sum is 0.
\end{lemma}

\begin{proof}
We prove this lemma by induction on $k$. If $k=2$, then by the recursion Lemma \ref{rec} and the equation \eqref{k=1}, we get
\begin{eqnarray*}
\sum_{u \in (\mathcal{O}_{K}/\mathfrak{d})^{*}}\psi(u)N_2(\mathfrak{n}, \mathfrak{e}, \mathfrak{g}, \mathfrak{d},u)&=&\sum_{\substack{u \in (\mathcal{O}_{K}/\mathfrak{d})^{*}}}\psi(u)\frac{\varphi(\mathfrak{n})}{\varphi(\mathfrak{e})\varphi(\mathfrak{g})}\sum_{\mathfrak{j}\mid \mathfrak{e}}\mu(\mathfrak{j})\sum_{\mathfrak{t}\mid \mathfrak{g}}\mu(\mathfrak{t})N_{1}(\mathfrak{n}, \mathfrak{j}, \mathfrak{t}, \mathfrak{d},u)\\
&=& \frac{\varphi(\mathfrak{n})}{\varphi(\mathfrak{e})\varphi(\mathfrak{g})}\sum_{\substack{u \in (\mathcal{O}_{K}/\mathfrak{d})^{*}\\u\equiv r \pmod{(\mathfrak{j}, \mathfrak{d})}}}\psi(u)\sum_{\mathfrak{j}\mid \mathfrak{e}}\mu(\mathfrak{j})\sum_{\substack{\mathfrak{t}\mid \mathfrak{g}\\\mathfrak{t}=1}}\mu(\mathfrak{t})\frac{\varphi(\mathfrak{n})}{\varphi([\mathfrak{j}, \mathfrak{d}])}\\
&=& \frac{\varphi(\mathfrak{n})^2}{\varphi(\mathfrak{e})\varphi(\mathfrak{g})} \sum_{\mathfrak{j}\mid \mathfrak{e}}\frac{\mu(\mathfrak{j})}{\varphi([\mathfrak{j}, \mathfrak{d}])} \sum_{\substack{u \in (\mathcal{O}_{K}/\mathfrak{d})^{*}\\u\equiv r \pmod{(\mathfrak{j}, \mathfrak{d})}}}\psi(u).
\end{eqnarray*}

By Lemma \ref{chr}, we have $$\sum_{\substack{u \in (\mathcal{O}_{K}/\mathfrak{d})^*\\ u\equiv r \pmod{(\mathfrak{j},\mathfrak{d})}}} \psi(u)\neq 0  \;\;\;\;\text{ if and only if }\;\;\;\; (\mathfrak{j},\mathfrak{d})=\mathfrak{d}, \;\; \text{ that is, if and only if }\;\; \mathfrak{d}\mid \mathfrak{j}.$$
Furthermore, if $\mathfrak{d} \mid \mathfrak{j}$, then $$\sum_{\substack{u \in (\mathcal{O}_{K}/\mathfrak{d})^*\\ u\equiv r \pmod{(\mathfrak{j},\mathfrak{d})}}} \psi(u) = \psi(r).$$
Hence
\begin{eqnarray*}
\sum_{u \in (\mathcal{O}_{K}/\mathfrak{d})^{*}}\psi(u)N_2(\mathfrak{n}, \mathfrak{e}, \mathfrak{g}, \mathfrak{d},u)&=& \frac{\varphi(\mathfrak{n})^2 \psi(r)}{\varphi(\mathfrak{e})\varphi(\mathfrak{g})} \sum_{\substack{\mathfrak{j}\mid \mathfrak{e}\\\mathfrak{d}\mid \mathfrak{j}}}\frac{\mu(\mathfrak{j})}{\varphi([\mathfrak{j}, \mathfrak{d}])}\\
&=& \frac{\varphi(\mathfrak{n})^2 \psi(r)}{\varphi(\mathfrak{e})\varphi(\mathfrak{g})} \sum_{\substack{\mathfrak{d}\mid\mathfrak{j}\mid \mathfrak{e}}}\frac{\mu(\mathfrak{j})}{\varphi(\mathfrak{j})}.
\end{eqnarray*}

Therefore, it is clear that, if $\mathfrak{d}\nmid \mathfrak{e}$ or if $\mathfrak{d}$ is not square-free, then the sum is empty. If $\mathfrak{d}\mid \mathfrak{e}$, then
\begin{eqnarray*}
\sum_{u \in (\mathcal{O}_{K}/\mathfrak{d})^{*}}\psi(u)N_2(\mathfrak{n}, \mathfrak{e}, \mathfrak{g}, \mathfrak{d},u)&=& \frac{\varphi(\mathfrak{n})^2 \psi(r) \mu(\mathfrak{d})}{\varphi(\mathfrak{e})\varphi(\mathfrak{g})}\prod_{\mathfrak{p}\mid \mathfrak{d}}\frac{1}{N(\mathfrak{p})-1}\prod_{\substack{\mathfrak{p}\mid \mathfrak{e}\\\mathfrak{p}\nmid \mathfrak{d}}}\left(1-\frac{1}{N(\mathfrak{p})-1}\right).
\end{eqnarray*}
Hence the formula is true for $k=2$. Assume that the formula is true for $k-1$. Then by the recursion Lemma \ref{rec}, we get
\begin{eqnarray*}
& &\sum_{u \in (\mathcal{O}_{K}/\mathfrak{d})^{*}}\psi(u)N_k(\mathfrak{n}, \mathfrak{e}, \mathfrak{g}, \mathfrak{d},u)\\
&=&\sum_{\substack{u \in (\mathcal{O}_{K}/\mathfrak{d})^{*}}}\psi(u)\frac{\varphi(\mathfrak{n})}{\varphi(\mathfrak{e})\varphi(\mathfrak{g})}\sum_{\mathfrak{j}\mid \mathfrak{e}}\mu(\mathfrak{j})\sum_{\mathfrak{t}\mid \mathfrak{g}}\mu(\mathfrak{t})N_{k-1}(\mathfrak{n}, \mathfrak{j}, \mathfrak{t}, \mathfrak{d},u)\\
&=& \frac{\varphi(\mathfrak{n})}{\varphi(\mathfrak{e})\varphi(\mathfrak{g})}\sum_{\mathfrak{j}\mid \mathfrak{e}}\mu(\mathfrak{j})\sum_{\substack{\mathfrak{t}\mid \mathfrak{g}\\(\mathfrak{t},\mathfrak{j})=1}}\mu(\mathfrak{t})\sum_{\substack{u \in (\mathcal{O}_{K}/\mathfrak{d})^{*}}}\psi(u)N_{k-1}(\mathfrak{n}, \mathfrak{j}, \mathfrak{t}, \mathfrak{d},u)\\
&=& \frac{\varphi(\mathfrak{n})}{\varphi(\mathfrak{e})\varphi(\mathfrak{g})}\sum_{\substack{\mathfrak{j}\mid \mathfrak{e}\\\mathfrak{d}\mid \mathfrak{j}}}\mu(\mathfrak{j})\sum_{\substack{\mathfrak{t}\mid \mathfrak{g}\\(\mathfrak{t},\mathfrak{j})=1}}\mu(\mathfrak{t})\frac{\varphi(\mathfrak{n})^{k-1} \psi(r)\mu(\mathfrak{d})^{k-2}}{\varphi(\mathfrak{j})\varphi(\mathfrak{t})}\prod_{\mathfrak{p}\mid \mathfrak{d}}\frac{1}{(N(\mathfrak{p})-1)^{k-2}}\\
&\times & \prod_{\substack{\mathfrak{p}\mid \mathfrak{j}\\\mathfrak{p}\nmid \mathfrak{d}}}\left(1- \frac{1}{N(\mathfrak{p})-1}+\cdots+\frac{(-1)^{k-2}}{(N(\mathfrak{p})-1)^{k-2}} \right) \prod_{\substack{\mathfrak{p}\mid \mathfrak{t}}}\left(1- \frac{1}{N(\mathfrak{p})-1}+\cdots+\frac{(-1)^{k-3}}{(N(\mathfrak{p})-1)^{k-3}} \right)\\
&=& \frac{\varphi(\mathfrak{n})^k \psi(r)\mu(\mathfrak{d})^{k-2}}{\varphi(\mathfrak{e})\varphi(\mathfrak{g})}\prod_{\mathfrak{p}\mid \mathfrak{d}}\frac{1}{(N(\mathfrak{p})-1)^{k-2}}\;\; \sum_{\substack{\mathfrak{j}\mid \mathfrak{e}\\\mathfrak{d}\mid \mathfrak{j}}}\frac{\mu(\mathfrak{j})}{\varphi(\mathfrak{j})}\;\;\prod_{\substack{\mathfrak{p}\mid \mathfrak{j}\\\mathfrak{p}\nmid \mathfrak{d}}}\left(1- \frac{1}{N(\mathfrak{p})-1}+\cdots+\frac{(-1)^{k-2}}{(N(\mathfrak{p})-1)^{k-2}} \right)\\
&\times & \sum_{\substack{\mathfrak{t}\mid \mathfrak{g}}}\frac{\mu(\mathfrak{t})}{\varphi(\mathfrak{t})}\;\; \prod_{\substack{\mathfrak{p}\mid \mathfrak{t}}}\left(1- \frac{1}{N(\mathfrak{p})-1}+\cdots+\frac{(-1)^{k-3}}{(N(\mathfrak{p})-1)^{k-3}} \right)\\
&=& \frac{\varphi(\mathfrak{n})^k \psi(r)\mu(\mathfrak{d})^{k-2}}{\varphi(\mathfrak{e})\varphi(\mathfrak{g})}\prod_{\mathfrak{p}\mid \mathfrak{d}}\frac{1}{(N(\mathfrak{p})-1)^{k-2}}\mu(\mathfrak{d})\prod_{\mathfrak{p}\mid \mathfrak{d}}\frac{1}{N(\mathfrak{p})-1}\\
&\times & \prod_{\substack{\mathfrak{p}\mid \mathfrak{e}\\\mathfrak{p}\nmid \mathfrak{d}}}1-\frac{1}{N(\mathfrak{p})-1}\left(1- \frac{1}{N(\mathfrak{p})-1}+\cdots+\frac{(-1)^{k-2}}{(N(\mathfrak{p})-1)^{k-2}} \right)\\
&\times & \prod_{\substack{\mathfrak{p}\mid \mathfrak{g}}}1-\frac{1}{N(\mathfrak{p})-1}\left(1- \frac{1}{N(\mathfrak{p})-1}+\cdots+\frac{(-1)^{k-3}}{(N(\mathfrak{p})-1)^{k-3}} \right)\\
&=& \frac{\varphi(\mathfrak{n})^k \psi(r)\mu(\mathfrak{d})^{k-1}}{\varphi(\mathfrak{e})\varphi(\mathfrak{g})}\prod_{\mathfrak{p}\mid \mathfrak{d}}\frac{1}{(N(\mathfrak{p})-1)^{k-1}} \prod_{\substack{\mathfrak{p}\mid \mathfrak{e}\\\mathfrak{p}\nmid \mathfrak{d}}}\left(1- \frac{1}{N(\mathfrak{p})-1}+\cdots+\frac{(-1)^{k-1}}{(N(\mathfrak{p})-1)^{k-1}} \right)\\
&\times & \prod_{\substack{\mathfrak{p}\mid \mathfrak{g}}}\left(1- \frac{1}{N(\mathfrak{p})-1}+\cdots+\frac{(-1)^{k-2}}{(N(\mathfrak{p})-1)^{k-2}} \right).
\end{eqnarray*}
This proves the result.
\end{proof}

{\noindent \bf Proof of Theorem \ref{MAIN-TH}.}
We now continue the evaluation of $M_{\chi,k}(\mathfrak{n})$. From the equation \eqref{sum} and Lemma \ref{prodprime}, we have
\begin{eqnarray*}
M_{\chi,k}(\mathfrak{n})&=&\sum_{\substack{\mathfrak{e}\mid \mathfrak{n}}}(\mu \ast f)(\mathfrak{e})\sum_{\substack{b_1,b_2,\ldots, b_s \in \mathcal{O}_{K}/\mathfrak{n}\\\mathfrak{e} \mid b_1,\ldots, \mathfrak{e} \mid b_s}}\;\;\sum_{\mathfrak{g}\mid \mathfrak{n}} \mu(\mathfrak{g})\sum_{u \in (\mathcal{O}_{K}/\mathfrak{d})^{*}}\psi(u)N_k(\mathfrak{n}, \mathfrak{e}, \mathfrak{g}, \mathfrak{d},u)\\
&=& \sum_{\substack{\mathfrak{e}\mid \mathfrak{n}\\\mathfrak{d}\mid \mathfrak{e}}}(\mu \ast f)(\mathfrak{e})\sum_{\substack{b_1,b_2,\ldots, b_s \in \mathcal{O}_{K}/\mathfrak{n}\\\mathfrak{e} \mid b_1,\ldots, \mathfrak{e} \mid b_s}}\;\;\sum_{\substack{\mathfrak{g}\mid \mathfrak{n}\\(\mathfrak{g},\mathfrak{e})=1}} \mu(\mathfrak{g})\frac{\varphi(\mathfrak{n})^k \psi(r)\mu(\mathfrak{d})^{k-1}}{\varphi(\mathfrak{e})\varphi(\mathfrak{g})}\prod_{\mathfrak{p}\mid \mathfrak{d}}\frac{1}{(N(\mathfrak{p})-1)^{k-1}} \\
&\times & \prod_{\substack{\mathfrak{p}\mid \mathfrak{e}\\\mathfrak{p}\nmid \mathfrak{d}}}\left(1- \frac{1}{N(\mathfrak{p})-1}+\cdots+\frac{(-1)^{k-1}}{(N(\mathfrak{p})-1)^{k-1}} \right) \prod_{\substack{\mathfrak{p}\mid \mathfrak{g}}}\left(1- \frac{1}{N(\mathfrak{p})-1}+\cdots+\frac{(-1)^{k-2}}{(N(\mathfrak{p})-1)^{k-2}} \right)\\
&=& \varphi(\mathfrak{n})^k \psi(r)\mu(\mathfrak{d})^{k-1} \prod_{\mathfrak{p}\mid \mathfrak{d}}\frac{1}{(N(\mathfrak{p})-1)^{k-1}}\;\; \sum_{\substack{\mathfrak{e}\mid \mathfrak{n}\\\mathfrak{d}\mid \mathfrak{e}}}\frac{(\mu \ast f)(\mathfrak{e})}{\varphi(\mathfrak{e})}\;\; \prod_{\substack{\mathfrak{p}\mid \mathfrak{e}\\\mathfrak{p}\nmid \mathfrak{d}}}\left(1- \frac{1}{N(\mathfrak{p})-1}+\cdots+\frac{(-1)^{k-1}}{(N(\mathfrak{p})-1)^{k-1}} \right)\\
&\times & \sum_{\substack{b_1,b_2,\ldots, b_s \in \mathcal{O}_{K}/\mathfrak{n}\\\mathfrak{e} \mid b_1,\ldots, \mathfrak{e} \mid b_s}}\;\; \sum_{\substack{\mathfrak{g}\mid \mathfrak{n}\\(\mathfrak{g},\mathfrak{e})=1}} \frac{\mu(\mathfrak{g})}{\varphi(\mathfrak{g})}\;\; \prod_{\substack{\mathfrak{p}\mid \mathfrak{g}}}\left(1- \frac{1}{N(\mathfrak{p})-1}+\cdots+\frac{(-1)^{k-2}}{(N(\mathfrak{p})-1)^{k-2}} \right)\\
&=& \varphi(\mathfrak{n})^k \psi(r)\mu(\mathfrak{d})^{k-1} \prod_{\mathfrak{p}\mid \mathfrak{d}}\frac{1}{(N(\mathfrak{p})-1)^{k-1}}\;\;\sum_{\substack{\mathfrak{e}\mid \mathfrak{n}\\\mathfrak{d}\mid \mathfrak{e}}}\frac{(\mu \ast f)(\mathfrak{e})}{\varphi(\mathfrak{e})}\;\;\prod_{\substack{\mathfrak{p}\mid \mathfrak{e}}}\left(1- \frac{1}{N(\mathfrak{p})-1}+\cdots+\frac{(-1)^{k-1}}{(N(\mathfrak{p})-1)^{k-1}} \right) \\
&\times & \prod_{\substack{\mathfrak{p}\mid \mathfrak{d}}}\left(1- \frac{1}{N(\mathfrak{p})-1}+\cdots+\frac{(-1)^{k-1}}{(N(\mathfrak{p})-1)^{k-1}} \right)^{-1} \sum_{\substack{b_1,b_2,\ldots, b_s \in \mathcal{O}_{K}/\mathfrak{n}\\\mathfrak{e} \mid b_1,\ldots, \mathfrak{e} \mid b_s}}\\
&\times & \prod_{\substack{\mathfrak{p}\mid \mathfrak{n}\\\mathfrak{p}\nmid \mathfrak{e}}}1-\frac{1}{N(\mathfrak{p})-1}\left(1- \frac{1}{N(\mathfrak{p})-1}+\cdots+\frac{(-1)^{k-2}}{(N(\mathfrak{p})-1)^{k-2}} \right)\\
&=& \varphi(\mathfrak{n})^k \psi(r)\mu(\mathfrak{d})^{k-1} \prod_{\mathfrak{p}\mid \mathfrak{d}}\frac{1}{(N(\mathfrak{p})-1)^{k-1}}\;\;\sum_{\substack{\mathfrak{e}\mid \mathfrak{n}\\\mathfrak{d}\mid \mathfrak{e}}}\frac{(\mu \ast f)(\mathfrak{e})}{\varphi(\mathfrak{e})}\;\;\prod_{\substack{\mathfrak{p}\mid \mathfrak{e}}}\left(1- \frac{1}{N(\mathfrak{p})-1}+\cdots+\frac{(-1)^{k-1}}{(N(\mathfrak{p})-1)^{k-1}} \right)\\
&\times & \prod_{\substack{\mathfrak{p}\mid \mathfrak{d}}}\left(1- \frac{1}{N(\mathfrak{p})-1}+\cdots+\frac{(-1)^{k-1}}{(N(\mathfrak{p})-1)^{k-1}} \right)^{-1} \sum_{\substack{b_1,b_2,\ldots, b_s \in \mathcal{O}_{K}/\mathfrak{n}\\\mathfrak{e} \mid b_1,\ldots, \mathfrak{e} \mid b_s}}\\
&\times &  \prod_{\substack{\mathfrak{p}\mid \mathfrak{n}}}\left(1- \frac{1}{N(\mathfrak{p})-1}+\cdots+\frac{(-1)^{k-1}}{(N(\mathfrak{p})-1)^{k-1}} \right)\;\;\prod_{\substack{\mathfrak{p}\mid \mathfrak{e}}}\left(1- \frac{1}{N(\mathfrak{p})-1}+\cdots+\frac{(-1)^{k-1}}{(N(\mathfrak{p})-1)^{k-1}} \right)^{-1}.
\end{eqnarray*}
Since $$\varphi_{k}(\mathfrak{n}) = \varphi(\mathfrak{n})^k \prod_{\mathfrak{p} \mid \mathfrak{n}} \left(1-\frac{1}{N(\mathfrak{p})-1}+\frac{1}{(N(\mathfrak{p})-1)^2}+\cdots+ \frac{(-1)^{k-1}}{(N(\mathfrak{p})-1)^{k-1}}\right),$$ we have

\begin{eqnarray*}
M_{\chi,k}(\mathfrak{n})&=&
\varphi_k(\mathfrak{n}) \psi(r)\mu(\mathfrak{d})^{k-1}\prod_{\mathfrak{p}\mid \mathfrak{d}}\frac{1}{(N(\mathfrak{p})-1)^{k-1}}\;\; \sum_{\substack{\mathfrak{e}\mid \mathfrak{n}\\\mathfrak{d}\mid \mathfrak{e}}}\frac{(\mu \ast f)(\mathfrak{e})}{\varphi(\mathfrak{e})} \sum_{\substack{b_1,b_2,\ldots, b_s \in \mathcal{O}_{K}/\mathfrak{n}\\\mathfrak{e} \mid b_1,\ldots, \mathfrak{e} \mid b_s}}\\ 
&&\prod_{\substack{\mathfrak{p}\mid \mathfrak{d}}}\left(1- \frac{1}{N(\mathfrak{p})-1}+\cdots+\frac{(-1)^{k-1}}{(N(\mathfrak{p})-1)^{k-1}} \right)^{-1}\\
&=& \mu(\mathfrak{d})^{k-1}\psi(r)\varphi\left(\frac{\mathfrak{n}_{0}^{k}}{\mathfrak{d}^{k-1}}\right)\varphi_k\left(\frac{\mathfrak{n}}{\mathfrak{n}_{0}}\right) \sum_{\substack{\mathfrak{d}\mid \mathfrak{e}\mid \mathfrak{n}\\ \mathfrak{e}\mid b_1,\ldots,\mathfrak{e}\mid b_s}}\frac{(\mu \ast f)(\mathfrak{e})}{\varphi(\mathfrak{e})}
\end{eqnarray*}

where $\mathfrak{n}_{0} \mid \mathfrak{n}$ is such that $\mathfrak{n}_0$ has the same prime ideal factors as that of $\mathfrak{d}$ and $\gcd\left(\mathfrak{n}_0, \frac{\mathfrak{n}}{\mathfrak{n}_0}\right) = 1$. This completes the proof of Theorem \ref{MAIN-TH}.
$\hfill\Box$


\medskip

{\bf Acknowledgements.} I would like to thank Ramakrishna Mission Vivekananda Educational and Research Institute, Belur Math for providing financial support.

%
%

\end{document}